\documentclass[sn-mathphys,Numbered]{sn-jnl}
\usepackage{graphicx}%
\usepackage{multirow}%
\usepackage{amsmath,amssymb,amsfonts}%
\usepackage{amsthm}%
\usepackage{mathrsfs}%
\usepackage[title]{appendix}%
\usepackage{xcolor}%
\usepackage{textcomp}%
\usepackage{manyfoot}%
\usepackage{booktabs}%
\usepackage{algorithmicx}%
\usepackage{algpseudocode}%
\usepackage{listings}%

\theoremstyle{thmstyleone}%
\newtheorem{theorem}{Theorem}
\newtheorem{lemma}{Lemma}
\newtheorem{corollary}{Corollary}
\newtheorem{algorithm}{Algorithm}

\newtheorem{proposition}[theorem]{Proposition}%

\theoremstyle{thmstyletwo}%
\newtheorem{example}{Example}%
\newtheorem{remark}{Remark}%

\theoremstyle{thmstylethree}%
\newtheorem{definition}{Definition}%

\begin{document}

\title[Magic Reverse Divisor]{Reverse Divisors and Magic Numbers}
\vspace{2cm}

\author{\sur{Eudes Antonio Costa} and \sur{Ronaldo Antonio Santos}\\
Federal University of Tocantins, Federal University of Goiás\\
{eudes@uft.edu.br}, {rasantos@ufg.br}}

\abstract{The study examines the relationship between Ball's magic numbers and reverses divisors. These numbers are the source of beautiful and curious properties. Activities related to numbers can be a fun way to motivate mathematics students,  while also enabling surprising analysis and connections.}

\keywords{Magic, Number, Reverse, Divisor}

\pacs[MSC Classification]{11A05, 97C70, 97F30}

\maketitle

\section{Introduction}\label{sec1}

This  paper presents the properties and connections between Ball’s magic numbers and their reverse divisors. Numbers are endless sources of curious and surprising problems and, different properties and connections appeared when we carefully analyzed these classes of numbers.

First, following \cite{costa4, webs1}, we present the definition of a Ball's magic numbers their properties and a characterization that allows identification with lists of numbers formed only by zero and one.

We presented Ball’s magic numbers as the first class, for example, $1089$. Such a number arises when, from a three-digit number, $572$, for example, we subtract its reverse $275$, obtaining the number $297$, which, added to its reverse $792$, results in $1089$. This sequence of calculations invariably results in $1089$, regardless of the three-digit number taken at the beginning.

Section 3 presents the definitions and properties of the reverse divisors. Reverse divisors are contained in a larger class of numbers called \textit{ permultiples}. Permutations of digits form numbers with multiples. Pairs $142857$ and $285714$ are examples of \textit{permultiples}, with $285714=2\times142857$ (see \cite{gugu}). In reverse divisors, the permutation must obtain a reverse number. Pairs $10989$ and $98901$ are  examples of this. The pair is a per multiple, as $98901=9\times10989$; moreover, the permutation obtained results in the reverse of the first number. In this case, we say that $10989$ is the reverse divisor of $98901$, also known as \textit{palintuples} (see \cite{Holt2}).

Magic numbers are reverse divisors. We presented the connections between these two classes in Section 4. We showed the numbers representing both Ball’s magic numbers and reverse divisors. The sum of magic numbers that resulted in reverse divisors. Section 5 discusses the relationship between perfect squares and reverse divisors.

\section{Magic numbers}

Recreational activities in which the participant is called to develop a sequence of calculations whose results are presented in advance are common in the literature. Magic numbers are an example of such activities. The details and properties of magic numbers were first introduced in \cite{ball1} and were followed in \cite{ball, beiler, costa3, costa2, costa4, webs1}. Following these references, we will briefly present magic numbers and their proprieties in this section, considering the representation at the base $10$.

Let be $x_n=a_{n-1}\dots a_1a_0$ a positive non-palindromic number with $n$ digits, and for $ n \geq 2 $, the number of $ n $ digits obtained by reversing the position of the digits of $ x_n $ is called the reverse number of $ x_n $ and is denoted by $ x_n'$; thus, $ x_n '= a_{0}a_1 \ldots a_{n-1} $.
Now, we consider the following algorithm:

\begin{algorithm}:
	\textbf{Ball's magic number}
	\label{AnmBb}
	\begin{enumerate}
		\item Let be a number  $x_n$;
		\item Write the reverse $x_n'$ ;
		\item Find the difference(positive) $y_n=|x_n-x_n'|$;
		\item Write the reverse $y_n'$;
		\item Add to obtain the number $B=y_n+y_n'$.
	\end{enumerate}
\end{algorithm}

\begin{remark}

Let be a number $x_n$ with $n$ digits; then, the numbers $x_n',\, y_n$ and $y_n'$ are (considered) numbers of $n$ digits, even if any digit is zero on the left.

\end{remark}

As defined in \cite{costa3,costa4}, the number $B\neq 0$ obtained from Algorithm \ref{AnmBb} is called the \textbf{Ball's magic number}.
When $x_n > x_n '$,  we can write the magic number of the {\it Ball} as $ B = (x_n - x_n') + (x_n - x_n ')' $.

\begin{example}
\label{E1089} 
For $n=2$, let be $x_2=71$ and $x_2'=17$, so $y_2=54$ and $y_2'=45$. We obtain $B=54+45=99$. For $n=3$, let be  $x_3=843$ and $x_3'=348$, so we have $y_3=843 - 348 = 495$. Finally $B = 495 +594= 1089$. Therefore, $99$ and $1089$ are Ball's magic numbers.
\end{example}

Some surprising results relate to Ball's magic numbers.
\begin{theorem}
	\label{PBb-1}
	
 Every non-zero magic number of {\it Ball} $B$ is a multiple of $99$.
\end{theorem}

 We provided the proof in Refs. Section~\ref{SQNB}.\\

See that the sequence of the 
numbers of possible end {\it Ball} numbers, corresponding to initial numbers of 2, 3, 4, 5, 6, 7, ... digits 
begins 1, 1, 3, 3, 8, 8, ... .
Thus, the number $B(n)$ of possible {\it Ball} numbers associated with the number $x_n$ with $n=2k+1 $ digits (or $n=2k$ digits) is the sum of the Fibonacci numbers
i.e.,

\begin{theorem}
	\label{TMb}
 Let be a natural number $x_{2n+1} $ with $2n+1$ digits; for all $n\geq 1$, the number of {\it Ball} $B(2n+1)$ is
\[
F_{2n}+ F_{2(n-1)}+\cdots + F_2 \ ,
\]
where $F_j$ is the Fibonacci number at the $j$ position.
\end{theorem}

This paper does not present the proof of Theorem~\ref{TMb}. Interested readers can refer to 
 demonstration in \cite{costa3, costa4, webs1}.

\subsection{Codes and Ball numbers}
\label{SQNB}

Webster \cite{webs1} showed an important characterization of Ball's magic numbers and, with this, proved Theorem \ref{PBb-1} and the relationship between Ball's magic numbers and Fibonacci numbers. The relationship we show between Ball's magic numbers and reverse divisors is also strongly linked to this characterization. Thus, following \cite{costa4, webs1}, we will summarize here this characterization.

Following our notation and Algorithm \ref{AnmBb}, let $x_{n+1}$ be a number with $n+1$ digits and $x_{n+1}'$ be its reverse.
So,
\begin{eqnarray*}
x_{n+1} & =& a_na_{n-1}\dots a_{n-i}\dots a_i\dots a_1 a_0 \ , \\
x_{n+1}'& = & a_0 a_1 \dots a_i \dots  a_{n-i} \dots  a_{n-1}a_n \ ,  
\end{eqnarray*}
with $a_n>a_0$.
We have
\begin{eqnarray*}
x_{n+1}-x_{n+1}' & =& (a_n-a_0)10^n+(a_{n-1}-a_1)10^{n-1}+\dots+(a_{n-i}-a_i)10^{n-i}+\\
& & + \dots+(a_i-a_{n-i})10^i+\dots+(a_1-a_{n-1})10+(a_0-a_n)\ .
\end{eqnarray*}

The representation at the base $10$ requires each coefficient to be greater than or equal to zero and less than $10$. In some cases, for the coefficient to be positive, you need to shift $10$ from the order $i$ to $i-1$ (the famous move of a ten or ``borrow one'').

Let us associate each number $x_{n+1}$ with $n+1$ digits to the
number $Z_{n+1}$ called the code of $x_{n+1}$. This code $Z_{n+1}$ comprises a sequence of $0's$ and $1's$ and has the necessary information to pass from the initial number $x_{n+1}$ to the resulting Ball's number $B$. \\

Initially, we presented the construction of $Z_{n+1}$ more intuitively.
We write the number $x_{n+1} = a_{n} \ldots a_0 $ with $a_{n} > a_0$, and below it, its reverse $x_{n+1}' = a_{0} \ldots a_{n}$, as
shown below:

\begin{center}
\begin{table}[h!]
\label{TB1}
\centering
\begin{tabular}{c c c c}
 & $a_{n}$ & $\ldots$ &  $a_0$ \\

-  & $a_0$ & $\ldots$ &  $a_{n}$ \\
\hline

 & * & $\ldots$ & *  

\end{tabular}

\end{table}
\end{center}

Consider the $i$th column from right to left $(i = 0, \ldots, n)$ by subtracting the number $x_{n+1}'$
of $x_{n+1}$. The digit $z_i$ is defined as follows: if the quantity $10$ needs to be regrouped from column $(i + 1)$th to $i$th, $z_i$ will be
1; otherwise, it is 0.
We write,
\[
\begin{array}{l}
x_{n+1}-x_{n+1}'  = \\ =(a_n-a_0-z_{n-1})10^n+(a_{n-1}-a_1-z_{n-2}+z_{n-1}10)10^{n-1} +\dots \\
  +(a_{n-i}-a_i-z_{n-i-1}+z_{n-i}10)10^{n-i}+\dots+(a_i-a_{n-i}-z_{i-1}+z_{i}10)10^i + \\
  \dots  +(a_1-a_{n-1}-z_0+z_1 10)10+(a_0-a_n+z_010)  \ .
\end{array}
\]

We obtain the strings $z_0, \ldots , z_{n}$ of $0's$ and $1's$. The number $z_0 \ldots z_{n}$ with $n+1$ digits is called the code of $x_{n+1}$ and is denoted by $Z_{n+1}$, that is, $Z_{ n+1}=z_0z_1 \ldots z_{n-1} z_{n}$. Because we assume the following:
$a_{n}> a_0$, $z_0 = 1$ and $z_{n} = 0$. The number of $n$ digits obtained from $Z_{n+1}$ by excluding the digit $z_{n}=0$ (at the end) is denoted by $Z_{\overline{n+1}}=z_0 \ldots z_{n-1}$ and is called the truncated code of $x_{n+1}$.

Formally we have:

\begin{definition}[code]
\label{Dcod}

Let be  $x_{n+1}=a_na_{n-1}\dots a_1a_0$, with $a_n>a_0$. The number  $Z_{n+1 }=z_0 \ldots z_{n}$ is called the code related to $x_{n+1}$ if $z_0=1$ and $z_n=0$, and recursively

\[
z_i =\begin{cases} 1;\;\textrm{ if  } a_i-a_{n-i}-z_{i-1}<0\ ;\\
0;\;\textrm{ if  } a_i-a_{n-i}-z_{i-1} \geq 0\ ;\\
\end{cases}
\]
 for $i=0,\dots, n-1$ and  $a_{-1}=0$.
\end{definition}

Notice that $0\leq a_i-a_{n-i}-z_{i-1}+z_i 10<10$ for $i=0,\dots,n-1$. 

Following code definition, we write    
\[
y_n=x_{n+1}-x_{n+1}'= \sum_{i=0}^n (a_i-a_{n-i}-z_{i-1}+z_{i}10)10^i \ .
\]

\begin{proof}{\it Theorem~\ref{PBb-1}}\\

By hypothesis, we have that $a_n > a_0$; therefore, $z_0=1$. With this notation, we can write the magic number of {\it Ball } in base $10$ as:

\begin{eqnarray}
\nonumber B & = & y_n+y_n' \\ 
\nonumber &=& \sum_{i=0}^n \bigl( (a_i-a_{n-i}-z_{i-1}+z_{i}10)+(a_{n-i}-a_{i}-z_{n-i-1}+z_{n-i}10) \bigr) 10^i \\
\nonumber &=& \sum_{i=0}^n (-z_{i-1}+z_{i}10-z_{n-i-1}+z_{n-i}10)10^i \\
\nonumber &=& -\sum_{i=0}^nz_{i-1}10^i+\sum_{i=0}^nz_{i}10^{i+1}-\sum_{i=0}^nz_{n-i-1}10^i+\sum_{i=0}^nz_{n-i}10^{i+1} \\
\nonumber & = & -\sum_{i=0}^{n-1}z_{i}10^{i+1}+\sum_{i=0}^{n-1}z_{i}10^{i+1}-\sum_{i=1}^nz_{n-i}10^{i-1}+10^2\sum_{i=1}^{n}z_{n-i}10^{i-1} \\
\nonumber &= & (10^2-1)\sum_{i=1}^{n}z_{n-i}10^{i-1} \\
\label{ECT} &= & 99\cdot (z_0z_1\dots z_{n-1}) \ .
\end{eqnarray} 
Therefore, the number of {\it Ball} $B$ is a multiple of $99$.
\end{proof}

Furthermore, all Ball numbers $B$ are multiples of the truncated code $z_0z_1\dots z_{n-1}$.
We provide several examples to illustrate these results.

\begin{example}

Let be a six-digit number $x_6 = 397862$. By subtracting $x_6' = 268793$ from $x_6$, we obtain

\vspace{-7mm}
\begin{center}
\begin{table}[h!]
\label{TB2}
\centering
\begin{tabular}{c c c c c c c}
 & 3& ${\not{9}}^{8}$& ${\not{7}}^{17}$& ${\not{8}}^{7}$& ${\not{6}}^{15}$ & ${\not{2}}^{12}$ \\
-  & 2 & 6 & 8& 7 & 9 & 3 \\
\hline
 & 1 & 2 & 9&0&6&9 \\
\hline \hline 
 & 0 & 0 & 1& 0 & 1 & 1
\end{tabular}

\end{table}
\end{center}
 
Notice that we regrouped ten (10) from the adjacent column on the left in columns 0, 1, and 3. Using the above notation, we have $z_0 =1 , \ z_1 = 1, \ z_2 = 0, \ z_3 = 1, \ z_4 = 0 \mbox{ and } z_5 = 0 \ .$

Thus, $Z_6 = 110100$ is the code associated with the number $x_6=397862$ and the truncated code is $Z_{\overline{6}} = 11010$, which is a divisor of the Ball's number $B$. In fact, for $x_6 = 397862$, we have Ball's number
\[
B = 129069 + 960921 = 1089990 = 99\times 11010 = 99 \times Z_{\overline{6}} \ .
\]
\end{example}

\begin{example}

Consider the four-digit number $x_4 = accd$ in base $10$ with $a>d$. By subtracting $x_4' = dcca$ from $x_4$, we obtain

\begin{center}
\begin{table}[h!]
\label{TB3}
\centering
\begin{tabular}{ c c c c c}
 & a& ${\not{c}}^{(c-1)+10}$& ${\not{c}}^{(c-1)+10}$ & ${\not{d}}^{d+10}$ \\
-  &  d& c & c & a \\
\hline
  & (a-d)&9&9&(d-a+10) \\
\hline \hline 
 & 0& 1 & 1 & 1
\end{tabular}
\end{table}
\end{center}
 We need to move $10$ from the adjacent column on the left in columns 0, 1, and 2. Using the above notation we have $ z_0 =1 , \ z_1 = 1, \ z_2 = 1, \mbox{ and } z_3 = 0 \ .$
 
The code is $Z_4 = 1110$ and the truncated code $Z_{\overline{4}} = 111$ is a divisor of $B$. Therefore, for any $x_4 = accd$, we obtain the Ball's number
\[
B = (x_4 - x_4') + (x_4 - x_4 ')'= 99\times 111 = 99 \times Z_{\overline{4}} \ .
\]
\end{example}

\begin{example}
According to \cite{costa2, costa4} we have Ball's magic numbers and their respective codes for $n \leq 6$ as follows:

\begin{tabular}{|l|l|l|l|l|}
\hline
\hline
 n&$B$&factorization & divisor of $B$ (truncated code) & code  \\
 \hline
 \hline
 2&99 & $99\times 1$  & 1&10  \\
 \hline
 \hline
 3&1089&$99\times 11$  & 11 &110\\
 \hline
 \hline
 4&9999&$99\times 101$  & 101 &1010\\
 \hline
 4&10890& $99\times 110$ & 110 &1100\\
 \hline
 4&10989& $99\times 111$ & 111 &1110\\
 \hline
 \hline
 5&99099& $99\times 1001$ & 1001 &10010\\
 \hline
 5&109890& $99\times 1110$ & 1110 &11100\\
 \hline
 5&109989& $99\times 1111$ & 1111 &11110\\
 \hline
 \hline
 6&1089990& $99\times 11010$ & 11010 &110100\\
 \hline
 6&1098900& $99\times 11100$ & 11100 &111000\\
 \hline
  6&1099890& $99\times 11110$ & 11110 &111000\\
 \hline
  6&999999& $99\times 10101$ & 10101 &101010\\
 \hline
 6&991089& $99\times 10011$ & 10011 &100110\\
 \hline
 6&990099& $99\times 10001$ & 10001 &100010\\
 \hline
 6&1099989& $99\times 11111$ & 11111 &111110\\
 \hline
 6&1090089& $99\times 11101$ & 11011 &110110\\
 \hline
 \hline
\end{tabular}
\\

From Theorem \ref{TMb} that the quantity $B(2n+1)$ of Ball's numbers is greater than $2$ for $n>1$.
\end{example}

With the codes defined, we proceed following Webster \cite{webs1} to characterize them.

\begin{proposition}
\label{Pcod}
If $a_n>a_0$ then $z_0=1$ and $z_n=0$.
\end{proposition}

Furthermore,

\begin{proposition}[\cite{costa4, webs1}]
\label{simetria}
For all $i=1,\dots,n-1$, we have:
\begin{enumerate}
\item[a)] If $z_{i+1}=1$ and $z_i=0$, then $z_{n-i-1}=0$.
\item[b)] If $z_{i+1}=0$ and $z_i=1$, then $z_{n-i-1}=1$.
\end{enumerate}
\end{proposition}
\begin{proof}
(a) We have $z_{i+1}=1$; thus, 
\begin{align*}
0 \leq a_{i+1}-a_{n-(i+1)}-z_{i+1-1} & = a_{i+1}-a_{n-i-1} - z_{i} \\
                               & \stackrel{z_i=0}{=} a_{i+1}-a_{n-i-1} \ .
\end{align*}

So, $a_{n-i-1}-a_{i+1} > 0$, and hence, $a_{n-i-1}-a_{i+1}-z_{n-i-2} \geq 0 $, that is, $z_{n-i-1}=0$.

(b)  For $z_{i+1}=0$ we have  $a_{i+1}-a_{n-i-1}-z_i\geq0$.

As $z_i=1$, we obtain
$a_{i+1}-a_{n-i-1}-1\geq 0$, that is, $a_{n-i-1}-a_{i+1}\leq -1$ resulting in $a_{n-i- 1}-a_{i+1}-z_{n-i-2}\leq -1-z_{n-i-2}<0$, and
$z_{n-i-1}=1$.
\end{proof}

\begin{proposition}

A list of zeros(0's) and ones(1's), $z_0\dots z_n$, satisfying Propositions \ref{Pcod} and \ref{simetria} is the code.
\end{proposition}
\begin{proof}

Consider the number $x_{n+1}=z_0\dots z_n$ formed by a list satisfying Propositions \ref{Pcod} and \ref{simetria}: We show that the code related to the number $x_{n+1}$ is  equal to list $z_0\dots z_n$.

From the definition,
\[
w_i =\begin{cases} 1;\;\textrm{ if  } z_{n-i}-z_{i}-w_{i-1}<0\ ;\\
0;\;\textrm{ if  } z_{n-i}-z_{i}-w_{i-1} \geq 0\ ;\\
\end{cases}
\]

As $z_0=1$ and $z_n=0$ (Proposition \ref{Pcod}), it follows that $w_0=1$ because $z_n-z_0-w{-1}<0$ and $w_n=0$ because $z_0-z_n-w_{-1}\geq 0$.

The digit $w_1$ can be either zero or one. We assume that $w_1=1$. In this case, $z_{n-1}-z_1-w_{0}<0$ or $z_{n-1}-z_1<1$. If $z_1=0$ then it is also $z_{n-1}$. However,  this is a contradiction, because $z_0=1$ and $z_1=0$ imply $z_{n-1}=1$ (Proposition \ref{simetria}). Thus, $z_1=1$.

On the other hand, if $w_1=0$ then $z_{n-1}-z_1-w_{0}\geq0$ or $z_{n-1}-z_1\geq 1$. This results implies that $z_1=0$.

Continuing this argument, we demonstrate that $w_i=z_i$ for all $i=1\dots n-1$. Therefore,  $x_{n+1}$ is the code.

\end{proof}

If the hypotheses of Proposition \ref{simetria} are not verified, the symmetrical element can assume a value of zero or one. The following observations highlight this fact.

\begin{remark}
\label{simetria2}
For $i=1,\dots,n-1$, we have:
\begin{enumerate}
\item[a)] If $z_{i+1}=1$ and $z_i=1$, then $z_{n-i-1}$ can be $0$ or $1$.
\item[b)] If $z_{i+1}=0$ and $z_i=0$, then $z_{n-i-1}$ can be $0$ or $1$.
\end{enumerate}

\end{remark}

Furthermore, for codes with odd-digit quantities, we have

\begin{remark}

Let $z_0...z_{n-1}z_{n}z_{n+1}...z_{2n}$ be a code. So $z_n=z_{n-1}$. If they are different, for example, $z_{n-1}=1$ and $z_{n}=0$, it follows from Proposition \ref{simetria} that $z_{2n-n}=z_n=1 $, which is contradictory. We obtain a similar contradiction when we assume $z_{n-1}=0$ and $z_{n}=1$.
\end{remark}

Now let us study whether some specific sequences of  0's (zeros)  and 1's (ones) are codes.

\begin{proposition} 
\label{PP2}

For every $m\geq 0$, list $1\underbrace{00\ldots00}_{m \mbox{ times}}10$ is a code with $m+3$ digits.
\end{proposition}
\begin{proof}

The list has the form $1\underbrace{00\ldots00}_{ m \mbox{ times}}10$. We denote by
    $$100\dots0010=a_0a_1a_2\dots a_{m+3}\ . $$
 The initial conditions $a_0=1$ and $a_{m+3}=0$ are satisfied. In the symmetry condition, for $a_i=0$ and $a_{i+1}=0$ the symmetric element $a_{m+3-i-1}$ can be zero or one and is always satisfied. The case $a_0=1$ and $a_1=0$ implies that $a_{m+2}=1$, the case $a_{m+2}=1$ and $a_{m+3}=0$ implies that $a_{0}=1$ and the case $a_{m+1}=0$ and $a_{m+2}=1$ implies that $a_{1}=0$. In all cases, the implications are true. We conclude that the symmetry conditions are satisfied, and the list is a code.

\end{proof}

\begin{corollary} 
\label{CNMP2}

If $m>0$ is an integer, then $B_m=99\times(10^{m}+1)$ is the Ball's magic number.
\end{corollary}
\begin{proof}

The result follows from Proposition \ref{PP2} and the proof of Theorem \ref{PBb-1} because  $10^{m}+1$ is a truncated code.
\end{proof}

\begin{proposition}
\label{PRU}

Let $n\geq1$ then, a list of 0's and 1's   of the form
$$a_0a_1\dots a_{n-1}a_n=\underbrace{11\ldots 11}_{ n \mbox{ times}}0$$ is the code.
\end{proposition}
\begin{proof}

 The first condition for a list of 1's and 0's to be a code is $a_0=1$ and $a_n=0$ are satisfied. We now examine the symmetry condition.
 For $a_{i+1}=1$ and $a_i=1$, the symmetric element $a_{n-i-1}$ can be zero or one. Therefore, you will always be satisfied. The case in which $a_{n-1}=1$ and $a_n=0$ implies that $a_{n-n}=a_0=1$. This condition is satisfied.
 We conclude that list $a_0a_1\dots a_n=11\dots110$ is code.
\end{proof}

\begin{corollary}
\label{CRu}
If $n>0$ is an integer, then $B_n=99\times\underbrace{11\ldots 11}_{ n \mbox{ times}}$ is the Ball's magic number.
\end{corollary}
\begin{proof}
We follow  Theorem \ref{PBb-1} and Proposition  \ref{PRU}.
\end{proof}

\begin{remark}

The truncated code $\underbrace{11\ldots 11}_{ n \mbox{ times}}$ formed only by digit one is known in the literature as a \textit{repunits} number and uses the notation $ R_n = \underbrace {11 \hdots 1}_{n \, \mbox {times}} $, $ n \geq 1 $.
Using decimal notation, we have the following general expression for \textit{repunits}:
\begin{equation*}
\label{ERb10}
R_n = \dfrac{10^n -1}{9}, \ \mbox{ for all } n \geq 1.
\end{equation*}
The properties related to \textit{repunits} $R_{n}$, are given in \cite{beiler, carvalho, costadouglas1, costadouglas2}.
\end{remark}

\begin{remark}
\label{est}

We obtain the characterization of Ball's magic numbers and the corresponding codes by applying  Algorithm \ref{AnmBb} to the numbers in the form $x_{n+1}=a_n\dots a_0$, with $a_n>a_0$. Algorithm \ref{AnmBb} can also be applied to the case  $a_n=a_0$. In this case, the list began  and ended at zero. More generally, from $x_{n+1}>x'_{n+1}$, with $a_n=a_0$, we have a list of the form $\underbrace {00 \hdots 0}_{j \, \mbox { times}}z_jz_{j+1}\dots z_{n-j}\underbrace {00 \hdots 0}_{j \, \mbox {times}}$, where $j$ is the smallest index which $a_j> a_{n-j}$. Thus, the list $z_j\dots z_{n-j}$ is the code, and we refer to the complete list $\underbrace {00 \hdots 0}_{j \, \mbox { times}}z_jz_{j+1}\dots z_{n-j}\underbrace {00 \hdots 0}_{j \, \mbox {times}}$ of the \textit{ extended code}.
\end{remark}

\begin{proposition}
\label{PCod}

Code $R=\underbrace{11\ldots 11}_{ n \mbox{ times}}0$, $n\geq 1$ can be decomposed by adding code $A$ and an extended code $C$.
\end{proposition}
\begin{proof}

We want to write $R$ as the addition of code $A$ and extended code $C$; that is, we want to determine the codes $a_0a_1...a_{n-1}a_n$ and $c_0c_1.. .c_{n-1}c_n$ such that
\[
\underbrace{11\ldots 11}_{ n \mbox{ times}}0=a_0a_1...a_{n-1}a_n+ c_0c_1...c_{n-1}c_n \ .
\]
We choose a list $A=a_0a_1 \dots a_{n-1}a_n$  with  $a_0=1$, $a_n=0$, $a_1=0$ and $a_{n-1}=1$, and also satisfy Propositions \ref{Pcod} and \ref{simetria} for the indices $j=2\dots n-2$. Code $A$ has the form:
\[
10a_2a_3\dots a_{n-2}10 \ .
\]

We consider the auxiliary list $b_0b_1 \dots b_{n-1}b_n$ obtained from code, $A$ changing it one by zero and zero by one. Our auxiliary list has the following form:
\[
01b_2 \dots b_{n-2}01 \ .
\]

This is not a code; however, we show that the inner part (suppressing the extremes)
\[
1b_2...b_{n-2}0 
\]
  
is a code. Let us denote it by $c_0c_1...c_{n-2}$, that is, $c_j=b_{j+1}$, $j=0,...,n-2$.

The first condition, $c_0=1$ and $c_{n-2}=0$ is satisfied.
We have demonstrated that this satisfies the symmetry condition.
If $c_{i+1}=1$ and $c_i=0$, we must show that $c_{(n-2)-i-1}=0$. In fact, $c_{i+1}=1$ and $c_i=0$ implies $b_{i+2}=1$ and $b_{i+1}=0$. This implies that $a_{i+2}=0$ and $a_{i+1}=1$. Because the list $a_i$ satisfies the symmetry condition, we have $a_{n-i-2}=1$. Thus, $b_{n-i-2}=0$. But $b_{n-i-2}=c_{n-i-3}$, so $c_{n-i-3}=0$. In a similar way, we prove symmetry in  other cases. Finally, the list $C=c_0c_1...c_{n-2}0$ is the extended code, and $A+C=R$. 
\end{proof}

\begin{example}
Let us decompose code $11110$ by following the steps in Proposition \ref{PCod}: Consider the list $a_0a_1a_2a_3a_4=10a_210 \ .$

To be a code, $a_2$ must equal the digit on the left: $a_2=0$. The code has the form:

$a_0a_1a_2a_3a_4=10010 .$
By changing one by zero and zero by one, we obtain
$b_0b_1b_2b_3b_4=01101 , $
which is not code. However, taking the central part, we have $c_0c_1c_2=110 $; therefore,  $c_0c_1c_20=1100$ is an extended code and 
 $A+C=10010+1100=11110 . $
\end{example}

\begin{example} 

For $n=5$, Proposition \ref{PCod} shows that
$11110 = 11100+10 $ or $11110 = 10010+1100 $.
\end{example}

\begin{definition}[undulating]

A natural number, started by $1$ and alternating $0's$
and $1's$ is an {\it undulating} number. We use the notation $uz(n)$ to indicate the {\it undulating} number with $n$ digits, and $UZ$ to indicate the set of {\it undulating} numbers.
\end{definition}

\begin{example} 
{\it Undulations} Number 
 $uz(2)=10$, $uz(3)=101$, $uz(4)=1010$, $uz(5) = 10101$ and
$uz(11)=10101010101$. 
\end{example} 

We can find the {\it undulating} numbers information in \cite{costarpm, Pick}.
In general, we have

\begin{proposition}
\label{PUZ}
For all $n\geq 1$ and $uz(n)\in UZ$, $11\times uz(n)$ is a code for $n$ even and a truncated code for $n$ odd.
\end{proposition}
\begin{proof}
If $n$ even. We have $11\times uz(n)= 10\times uz(n)+uz(n)= 10\times R_n .$ It follows from Proposition \ref{PRU} that a list of forms $R=10\times R_n$ is code.

If $n$ is odd, then $11\times uz(n)= 10\times uz(n)+uz(n)=uz(n+1)+uz(n)=R_n, $ addition,  from Proposition \ref{PRU} we have $R_n$ as truncated code.
\end{proof}

\begin{example}
Some examples of Proposition \ref{PUZ}: $11\times uz(1)=11$; $11\times uz(2)=110$; $11\times uz(3)=1111$; $11\times uz(4)=11110$ e $11\times uz(5)=111111$.
\end{example}

In particular, we have 
\begin{proposition}
\label{PUZd}
For all $n>1$ even, the number $uz(n)\in UZ$ is code.
\end{proposition}
\begin{proof}

 For $n=2k$, we denote by $uz(2k)=z_0z_1\dots z_{2k-1}$. Note that one appears at even positions and zero at odd positions. The first condition to be a code $z_0=1$ and $z_{2k-1}=0$
 (Proposition \ref{Pcod}) is satisfied. To verify the symmetry condition (Proposition \ref{simetria}), given $i$ even, $z_i=1$ and $z_{i+1}=0$ must imply $z_{2k-1-i-1}= 1$. However, this is true because $2k-i-2$ is even. If $i$ is odd, then $z_i=0$ and $z_{i+1}=1$,  which implies $z_{2k-1-i-1}=0$. This condition is also verified because $2k-i-2$ is odd.
 We conclude that $uz(2k)$ is code.

\end{proof}

\begin{corollary}

Let be $n>0$ an integer; if $n$ is odd, then  $uz(n)$ is a truncated code.
\end{corollary}
\begin{proof}
This result follows from Proposition \ref{PUZd}.
\end{proof}

The following proposition and its corollary present the relationship between Ball numbers generated by numbers $x_{n+1}=a_n\dots a_0$, with $a_n=a_0$ and the extended codes.

\begin{proposition}[\cite{costa2}]
\label{PBD}

If $B_{0}$ is a Ball number generated by $x_{n}$ with $n$ digits, then $10B_{0}$ is the Ball number generated by $x_{n+2} $ with $n+2$ digits.
\end{proposition}
\begin{proof}

Given a number $x_{n} = a_{n-1} \ldots a_0 $ with $n $ digits, it follows from Theorem \ref{PBb-1} that the Ball number $B_0$ generated by the number $x_{n }$ is $99 \times Z_{\overline{n}}$, where $Z_{\overline{n}}$ is the truncated code associated with the number $x_{n}$, whose code we will denote by $Z_n=z_0z_1 \dots z_n$.

To obtain a new magic number, we add the same digit $a\not=0$ at the beginning and end of the number $x_{n}$ and obtain $x_{n+2}=aa_ {n-1} \ldots a_0 a $ with $n+2$ digits.

The extended code generated by $x_{n+2}$ is $Z_{n+2}=0z_0\cdots z_n0$, truncates which results in $Z_{\overline{n+2}}=z_0\cdots z_n$. Again, from Theorem \ref{PBb-1} it follows that the Ball number associated with the number $x_{n+2}$ is $99 \times Z_{\overline{n+2}}= 99 \times 10\times Z_ {\overline{n}}=10B_0$.
\end{proof}

From Proposition \ref{PBD} that
\begin{corollary}
\label{C10C}

If $z_0\cdots z_n$ is a code, then $z_0\cdots z_n0$ is an extended code.
\end{corollary}

\section{Reverse divisor}
\label{SDR}

This section discusses the  relationship between this number and its permutations.

\begin{definition}[permulpiple]
\label{Dper}

We call a number and one of its permutations \textit{permultiple} if one is a multiple of the other.
\end{definition}

\begin{example}[\cite{gugu}]
\label{E142857} 
The number $142857$ when multiplied by two, three, four, five, and six, results in another number with the same digits cyclically permuted:
\begin{eqnarray*}
142857 \cdot 2 \ &= \ 285714 \\ 
142857 \cdot 3\  &= \ 428571 \\ 
142857 \cdot 4 \ &= \ 571428 \\ 
142857 \cdot 5 \ &= \ 714285 \\ 
142857 \cdot 6 \ &= \ 857142 .
\end{eqnarray*}
 
Therefore, a pair of numbers formed by $142857$ and any one of the list: $285714$, $428571$, $571428$, $714285$ or $857142$ is an example of a \textit{permultiple} number.
\end{example}

\begin{example}[\cite{Holt, Holt2, webs2}]
\label{EP} 

 Numbers $1089$ and $9801$ are 
 \textit{permultiple} numbers, because $9801=9\cdot 1089$. Similarly, $10989$ and $98901$ are \textit{permultiple} numbers because $98901 = 9 \cdot 10989$.
\end{example}

\begin{example}[\cite{Augusto}]
\label{ET} 

The numbers $102564$ and $410256$ are also \textit{permultiple}, as $410256=4 \cdot 102564$.
\end{example}

Holt\cite{Holt, Holt2} studied several properties of a class of \textit{permultiple} numbers, the \textit{palintuples} (palindromes and multiples), in base $10$, as in Example~\ref{EP}; we presented examples in other bases. Holt\cite{Holt2} used the name \textit{permultiple} and is a juxtaposition of permutations and multiple words. The palindromic number is equal to the reverse. 
\textit{Palintuple} numbers are also known as \textit{reverse divisors} because their digits are arranged in reverse order.

In this section, we present (and study) a subclass of \textit{permultiple} numbers: \textit{palintuple} or \textit{reverse divisors}. In addition to definition of permultiple, 

\begin{definition}[reverse divisor]
\label{Ddr}

Let $x_n$ be a non-palindromic integer with $n\geq 2$ digits. We say that $x_n$ is a reverse divisor if there is an integer $1< k <10$ such that $x_n'=k\cdot x_n$; that is, $x_n$ divides its reverse $x_n'$. In this case, $x_n'$ is an inverse multiple of $x_n$.
\end{definition}

We often say that $x_n$ and $x_n'$ are the reverse divisors.

\begin{example}[\cite{Holt, Holt2, webs2}]
\label{Ewebs} 

The number $1089$ is a reverse divisor in the permultiple class because it divides the number $9801$, because $9801=9\cdot 1089$. Likewise, the numbers $10989$ and $98901$ are reverse divisors because $98901 = 9 \cdot 10989$.
\end{example}

We have

\begin{theorem}[\cite{webs2}]
\label{Tnm}
The value $10\underbrace{99\ldots 99}_{n-4 \mbox{ times}}89$ with $n>3$ is the reverse divisor of $98\underbrace{99\ldots 99}_{n-4\mbox{ times} }01$.
\end{theorem}

\begin{example}[\cite{Holt, Holt2, webs2}]
\label{Exdnm}

The numbers $21978$ and $87912$ are reverse divisors, because $87912 = 4 \cdot 21978$. Additionally, $2178$ and $8712$ are reverse divisors.
\end{example}

We also have

\begin{theorem}[\cite{webs2}]
\label{Tdnm}  The value
$21\underbrace{99\ldots 99}_{n-4 \mbox{ times}}78$, where $n>3$ is a reverse divisor of $87\underbrace{99\ldots 99}_{n -4\mbox{ times} }12$.
\end{theorem}

We present the proofs of Theorems~\ref{Tnm} and ~\ref{Tdnm} in Subsections below. 
Webster \cite{webs2} and others show that the two classes indicated in the above theorems contain all reverse divisors.

First, however, we call attention to the fact that the examples of reverse divisors presented in Examples~\ref{Ewebs} and~\ref{Exdnm} have more than $4$ digits.

The following two results guarantee no reverse divisors with $2$ or $3$ digits.

\begin{proposition}
There was no reverse divisor with $2$ digits.
\end{proposition}
\begin{proof}

We assume that  $x_2=ab$, where $a,b \in \{ 1, 2,3,4,5,6,7,8,9\}$ is a reverse divisor. Without loss of generality, we can assume $b>a$, that is, $b=a+c$ with $c\not=0$. So we have 
$x_2=10\cdot a + b \ \mbox{ and } \ x_2'=10\cdot b + a \ .$

Because $x_2$ is a divisor of $x_2'$, there exists an integer $1<k<10$ such that
\[
k=\dfrac{10\cdot b + a}{10\cdot a + b}=\dfrac{10\cdot b+100\cdot a -99\cdot a}{10\cdot a + b}=10- \dfrac{99\cdot a}{10\cdot a+b}=10-\dfrac{99\cdot a}{11\cdot a+c}.
\]
So,
\[
10-k=\dfrac{99\cdot a}{11\cdot a+c} \,\,\, \ \Rightarrow \,\,\,  (10-k)c=11\bigl( 9a-a(10-k)\bigr) \ .
\]
Then, $11$ divides $10-k$ or $11$ divides $c$. Contradiction.

 Therefore, $x_2$ is not the reverse divisor.

\end{proof}

In the same way

\begin{proposition}
There was no reverse divisor with $3$ digits.
\end{proposition}
\begin{proof}

Let $ a,b, c,d \in \{1, 2,3,4,5,6,7,8,9\}$ such that the number $x_3=abc$ is a reverse divisor with $c= a+d$. We have 
$
x_3=100\cdot a + 10\cdot b +c \ \mbox{ and } \ x_3'=100\cdot c + 10 \cdot b + a \ .
$\\
Because $x_3$ is a divisor of $x_3'$, there exists an integer $1<k<10$ such that
\[
k=\dfrac{100\cdot c +10\cdot b + a}{100\cdot a +10\cdot b + c}= \dfrac{100\cdot (a+d) +10\cdot b + (c-d )}{100\cdot a +10\cdot b + c}=1+\dfrac{99\cdot d}{100\cdot a +10\cdot b + c}\ ,
\]
that is,
\[
k-1=\dfrac{99\cdot d}{abc} ,\ \mbox{ so } \ (k-1)\cdot abc=9 \cdot 11\cdot d \ .
\]

Because $k-1<9$, it follows that neither $11$ nor $9$ divide $k-1$; therefore, $11$ and $3$ must divide $abc$ because $k-1$ can be a multiple of $3$. By inspection, we find no reverse divisors between multiples of $33$.

 Therefore, $x_3$ is not the reverse divisor.

\end{proof}

\subsection{Proof of Theorem~\ref{Tnm}}
\label{Snm}

Theorem~\ref{Tnm} is a direct consequence of the following result.

\begin{proposition}
\label{Pnm}

For every natural $n>3$ we obtain $
11\times(10^{n-2}-1)=10\underbrace{99\ldots 99}_{n-4 \mbox{ times}}89 \ .
$
\end{proposition}
\begin{proof}

We based the proof  on the induction of $n$. For $n=4$ we have
$11\times (10^2-1) = 1089$.

We assume that the result is true for some natural $k>4$; that is,
\begin{equation}
\label{Enm}
11\times(10^{k-2}-1)=10\underbrace{99\ldots 99}_{k-4 \mbox{ times}}89 \ .
\end{equation}
For $k+1$, we have
\begin{align}
\nonumber 11\times(10^{(k+1)-2}-1) & =  11\times(10\times 10^{k-2}-10+10-1)&\\
\label{Enm1} & =  11\times[(10\times (10^{k-2}-1)+(10-1)]&\\
\nonumber & =  10\times(11\times (10^{k-2}-1)+11\times 9 \ .&
\end{align}
By using the induction hypothesis, that is, Equation~\eqref{Enm} in Equation~\eqref{Enm1}, we have
\begin{eqnarray*}
11\times(10^{(k+1)-2}-1) & = & 10\times(10\underbrace{99\ldots 99}_{k-4 \mbox{ times}}89)+11\times 9\\
&=& 10\underbrace{99\ldots 99}_{k-4 \mbox{ times}}890+99\\
&=& 10\underbrace{99\ldots 99}_{k-4 \mbox{ times}}989\\
&=& 10\underbrace{99\ldots 99}_{(k+1)-4 \mbox{ times}}89\ .
\end{eqnarray*}
Which completes the proof.

\end{proof}

\begin{proposition}
\label{Pnr}

For every natural $n>3$ we obtain $
9\times 10\underbrace{99\ldots 99}_{n-4 \mbox{ times}}89 = 98\underbrace{99\ldots 99}_{n-4 \mbox{ times}}01 \ .$
\end{proposition}
\begin{proof}

We based the proof on the induction of $n$. For $n=4$ we have
$9\times 1089 = 9801$.

From Proposition~\ref{Pnm} we have $11\times(10^{n-2}-1)=10\underbrace{99\ldots 99}_{n-4 \mbox{ times}}89$. Now, we admit that the result is valid for some natural $k>4$; that is,
\begin{equation}
\label{Enr}
9\times 11\times(10^{k-2}-1) = 98\underbrace{99\ldots 99}_{k-4 \mbox{ times}}01 \ .
\end{equation}
For $k+1$, we have
\begin{align}
\nonumber 9\times 11\times(10^{(k+1)-2}-1) & =  9\times 11\times(10\times 10^{k-2}-10+10-1)& \\
\label{Enr1} & =  9\times11\times[(10\times (10^{k-2}-1)+(10-1)]&\\
 \nonumber& =  10\times9\times(11\times (10^{k-2}-1)+9\times11\times 9\ . &
\end{align}
Using the induction hypothesis, Equation~\eqref{Enr}, in Equation~\eqref{Enr1}, we have
\begin{eqnarray*}
9\times 11\times(10^{(k+1)-2}-1) & = & 10\times(98\underbrace{99\ldots 99}_{k-4 \mbox{ times}}01) +9\times 11\times 9\\
&=& 98\underbrace{99\ldots 99}_{k-4 \mbox{ times}}010+891\\
&=& 98\underbrace{99\ldots 99}_{k-4 \mbox{ times}}901\\
&=& 98\underbrace{99\ldots 99}_{(k+1)-4 \mbox{ times}}901\ .
\end{eqnarray*}
Which completes the proof.

\end{proof}

Let us look into
\begin{proof}{\it of Theorem~\ref{Tnm}}\\
For $n=4$,  $9801=9\cdot 1089$ (Example~\ref{Ewebs}).

For $n>4$ let $Y=11\times(10^{n-2}-1)$: from Proposition~\ref{Pnm}, we have $Y=10\underbrace{99\ldots 99}_ {n-4 \mbox{ times}}89$. Because $Y'=98\underbrace{99\ldots 99}_{n-4 \mbox{ times}}01$, it follows from Proposition~\ref{Pnr} that $Y'=9\cdot Y$.

\end{proof}

\subsection{Proof of Theorem~\ref{Tdnm}}
\label{Sdnm}

The proof of Theorem~\ref{Tdnm} is a direct consequence of the next results.

\begin{proposition}
\label{Pdnm}

For every natural $n>3$ we obtain $
2\times 11\times(10^{n-2}-1)=21\underbrace{99\ldots 99}_{n-4 \mbox{ times}}78 \ .$
\end{proposition}
\begin{proof}

We again apply induction on $n$. For $n=4$ we have $2\times 11\times (10^2-1) = 2178$.

From Proposition~\ref{Pnm} we have $11\times(10^{n-2}-1)=10\underbrace{99\ldots 99}_{n-4 \mbox{ times}}89$, and admit that the result is valid for some natural $k>4$; that is,
\begin{equation}
\label{Ednm}
2\times 11\times(10^{k-2}-1)=2\times 10\underbrace{99\ldots 99}_{k-4 \mbox{ times}}89 = 21\underbrace{99\ldots 99}_{k-4 \mbox{ times}}78 \ .
\end{equation}

For $k+1$, we have

\begin{eqnarray}
\nonumber 2\times 11\times(10^{(k+1)-2}-1) & = & 2\times 11\times(10\times 10^{k-2}-10+10-1)\\
\label{Ednm1}   & = & 10 [2\times(11\times (10^{k-2}-1)]+2\times 11\times 9 \ .    
\end{eqnarray}

Using the induction hypothesis, Equation~\eqref{Ednm} in Equation~\eqref{Ednm1}, we have

\begin{eqnarray*}
2\times 11\times(10^{(k+1)-2}-1) & = & 10[ 2\times 10\underbrace{99\ldots 99}_{k-4 \mbox{ times}}89]+2\times 11\times 9\\
&=& 21\underbrace{99\ldots 99}_{k-4 \mbox{ times}}780+198\\
&=& 21\underbrace{99\ldots 99}_{k-4 \mbox{ times}}978\\
&=& 21\underbrace{99\ldots 99}_{(k+1)-4 \mbox{ times}}78\ .    
\end{eqnarray*}

Which completes the proof.
\end{proof}

\begin{proposition}
\label{Pdnr}

For every natural $n>3$ we obtain $
4\times 21\underbrace{99\ldots 99}_{n-4 \mbox{ times}}78 = 87\underbrace{99\ldots 99}_{n-4 \mbox{ times}}12 \ .
$
\end{proposition}
\begin{proof}

Again, we will do this is performed by inducing  on $n$. For $n=4$ we have $4\times 2178 = 8712$.

From Proposition~\ref{Pdnm}, we obtain $2\times 11\times(10^{n-2}-1)=21\underbrace{99\ldots 99}_{n-4 \mbox{ times}}78 $. Now, admit that the result is true for some natural $k>4$, that is,
\begin{equation}
\label{Ednr}
4\times 21\underbrace{99\ldots 99}_{k-4 \mbox{ times}}78 = 87\underbrace{99\ldots 99}_{k-4 \mbox{ times}}12 \ .
\end{equation}
For $k+1$, we have
\begin{eqnarray}
\nonumber 4\times 21\underbrace{99\ldots 99}_{(k+1)-4 \mbox{ times}}78 &=&
\label{Ednr1} 4\times 11\times(10^{(k+1)-2}-1) \\
 &=& 4\times 2[\times 11\times(10\times 10^{k-2}-10+10-1)]\\
   \nonumber&=& 10\times [2\times(11\times (10^{k-2}-1)]+4\times 2\times 11\times 9 \ .    
\end{eqnarray}
Using the induction hypothesis, Equation~\eqref{Ednr}, in Equation~\eqref{Ednr1}, we have
\begin{eqnarray*}
4\times 21\underbrace{99\ldots 99}_{(k+1)-4 \mbox{ times}}78 & = & 10\times(87\underbrace{99\ldots 99}_{k-4 \mbox{ times}}12)+792\\
&=& 87\underbrace{99\ldots 99}_{k-4 \mbox{ times}}120 + 792\\
&=& 87\underbrace{99\ldots 99}_{k-4 \mbox{ times}}912\\
&=& 87\underbrace{99\ldots 99}_{(k+1)-4 \mbox{ times}}12\    
\end{eqnarray*}
completes the proof.
\end{proof}

We can now prove Theorem~\ref{Tdnm}.
\begin{proof}{\bf of Theorem~\ref{Tdnm}}\\

If $n=4$,  $2178$ is the reverse divisor of $8712$ (Example~\ref{Exdnm}).

For $n>4$ let $Y=2\times 11\times(10^{n-2}-1)$. From, Proposition~\ref{Pdnm}, we have $Y= 21\underbrace{99\ldots 99}_{n-4 \mbox{ times}}78$. Because $Y'=87\underbrace{99\ldots 99}_{n-4 \mbox{ times}}12$, it follows from Proposition~\ref{Pdnr} that $Y'=4\cdot Y$.

\end{proof}

\section{Magic reverse divisors}
\label{SDRM}
. 
In this section, we identify the numbers that are both Ball's magic numbers and reverse divisors.

\begin{theorem}
\label{TDRM}
For every $n> 1$, the reverse divisor of $11\times(10^{n}-1)$ is a Ball's number.
\end{theorem}
\begin{proof}

For all $n>1$, we have that
\[
11\times(10^{n}-1)=11\times \underbrace{99\ldots 99}_{ n \mbox{ times}} = 11\times 9 \times \underbrace{11\ldots 11}_{ n \mbox{ times}} = 99\times \underbrace{11\ldots 11}_{ n \mbox{ times}} \ .
\]

 From Proposition \ref{PRU} we have $\underbrace{11\ldots 11}_{ n \mbox{ times}}0$ as the code. Thus, the list $\underbrace{11\ldots 11}_{ n \mbox{ times}}$ is truncated code. Therefore,  $99\times \underbrace{11\ldots 11}_{ n \mbox{ times}}$  is the Ball's magic number.
\end{proof}

From Theorem \ref{TDRM} and the results presented in Section \ref{Snm}, we find that each reverse divisor of form $11\times(10^n-1)$ has $9$ as the reverse quotient.

The following result from Theorem \ref{TDRM}:

\begin{corollary}

If $B$ is a magic number and reverse divisor, then $B$ added to its reverse $B'$ is $10B$.
\end{corollary}
\begin{proof}
From Theorem \ref{TDRM}, we obtain that $B=99\times \underbrace{11\ldots 11}_{ n \mbox{ times}} = 11\times(10^{n}-1) \ .
$
Additionally, from Proposition \ref{Pnr} we have $B'=9\times 11\times(10^{n}-1)$ which completes the proof.
\end{proof}

\begin{corollary}
The reverse divisor $22 \times (10^{n}-1)$, $n> 1$, is a magic number multiplied by $2$.
\end{corollary}
\begin{proof}
Note that $ 22 \times (10^{n}-1)=2\times 11 \times (10^{n}-1) \ .
 $
\end{proof}
 
The following results show that it is possible to write reverse divisors by adding two other magic numbers.

\begin{theorem}

Let be a reverse divisor $D=11\times(10^{n}-1)$ with $n\geq 2$. There are two magic numbers, $B_1$ and $B_2$, such that $D=B_1+ B_2 $.
\end{theorem}
\begin{proof}

It follows from Theorem \ref{TDRM} that the reverse divisor is $D= 99\times \underbrace{11\ldots 11}_{ n \mbox{ times}}$. The corresponding code was $R= \underbrace{11\ldots 11}_{ n \mbox{ times}}0$. From Proposition \ref{PCod} we can decompose code $R$ as the sum of code $A$ and the extended code $C$; that is, $R=A+C$.
In this manner, the reverse divisor $D$ is the sum of magic numbers because
\begin{eqnarray*}
D &=& 99\times \underbrace{11\ldots 11}_{n \mbox{ times}}= 99 \times R \\
&=& 99 \times (A+C) = 99 \times A+ 99 \times C \ .
\end{eqnarray*}
From Theorem \ref{PBb-1} we obtain $B_1=99 \times A$ and $B_2= 99 \times C$ as Ball's magic numbers.
\end{proof}

\begin{example}
We can write the reverse divisor $D=1099989= 99 \times 11111$ as the sum of $B_1=990099=99\times 10001$ and $B_2=10989=99\times 1110$, both of which are magic numbers.
\end{example}

\begin{example} 
\label{EDNM}

The reverse divisor $D=109 98 9 = 99 \times 1111$ has a code $11110 = 11100+10 \ .$
So, 
\begin{eqnarray*}
109 98 9 &= & 99 \times 1111 = 99 \times (1110+1) \\
 &= & 99 \times 1110 +99 \times 1 = 10989 0 +99 \ .
\end{eqnarray*}
Alternatively, we can write the code as $11110 = 10010+1100 $. So,
\begin{eqnarray*}
109 98 9 &= & 99 \times 1111 = 99 \times (1001+110) \\
 &= & 99 \times 1001 +99 \times 110 = 99 0 99 + 1089 0 \ .
\end{eqnarray*}
\end{example}

\begin{example}
\label{EMult}
We also have $1001001\times 111=111111111$, so $99\times 1001001\times 111=99\times 111111111$. It follows that the reverse divisor $99\times 111$ and magic number $99\times 1001001$ are divisors of the  reverse divisor $99\times 111111111$. We can give a similar fact by $101\times 11=1111$ because this leads to $99\times 101\times 11=99\times 1111$.
\end{example}

Example \ref{EMult} allows us to conclude that by factoring repunits into codes or extended codes, we obtain magic, a divisor of reverse divisors. The following result provides reverse divisors with a special type of divisor:

\begin{theorem}
\label{reversomagico}
  A reverse divisor $D=11\times(10^{2^n}-1)$ with $n\geq 2$ has at least $n+1$ magic numbers as divisors.
\end{theorem}
    \begin{proof}

 Factoring $D$, we get
    
    \begin{eqnarray*}
    D=11\times10^{2^n}-1&=&11(10-1)(10+1)(10^2+1)(10^4+1)...(10^{2^{n-1}}+1)\\
    &=&99(10+1)(10^2+1)(10^4+1)\hdots (10^{2^{n-1}}+1)
    \end{eqnarray*}

It follows from Proposition \ref{PP2} that  $(10^{2^n}-1)$ is the code for every $m>0$. Therefore,  from Corollary \ref{CNMP2}, every $B_m=11\times(10^{2^m}-1)$ is a divisor of $D$ to $0\leq m\leq n-1$. We conclude that magic numbers of the form $B_m=99\times(10^{2^{m}}+1)$ with $0\leq m\leq n-1$ are the divisors of $D$. Because $D$  is a magic number, we have at least $n+1$ divisors.
\end{proof}

\begin{example}

The number $1099999989=11\times(10^8-1)$ is divisible by the magic numbers: $990099=99\times(10^4+1)$, $9999=99\times(10^2+1)$ and $1089 =99\times(10+1)$.
\end{example}

The best magic number was 1089. It is a reverse divisor; some multiples are also reverse divisors, as shown below.

\begin{proposition}
\label{PUZm}

For all $n>1$ and $uz(n)\in UZ$, $X_n=1089\times uz(n)$ is the Ball's  magic number.
\end{proposition}
\begin{proof}

For all $n>1$, $1089\times uz(n)= 99\times 11\times uz(n) .$

From Proposition \ref{PUZ} the factor $11\times uz(n)$ is truncated or extended code. In both cases, Theorems \ref{PBb-1} or Remark \ref{est} guarantee that $1089\times uz(n)$ is the magic number.
\end{proof}

\begin{theorem}
\label{TUZ}

For all odd $n>1$, $X_n=1089\times uz(n)$ is the reverse divisor.
\end{theorem}
\begin{proof}

Thus, $X_n=99\times11\times uz(n)$. Proposition \ref{PUZ} indicates that $11\times uz(n)$ is one repunit. From Theorem \ref{Tnm} and Proposition \ref{Pnm} follows this result.
\end{proof}

\begin{corollary}

For all odd $n>1$, $Y_n=2178\times uz(n)$ is the reverse divisor.
\end{corollary}
\begin{proof}

From Theorem \ref{Tdnm} it follows that $2178$ is the reverse divisor. As $2178=2\times 1089$, it follows from Theorem \ref{TUZ} that $Y_n$ is also a reverse divisor.
\end{proof}

\section{Reverse divisor and square number}
\label{SDRQM}

We use $R_n$, $n\geq 1$, for a list with only ones or repunit,  and $X$ or $Y$ for any code (list formed by ones and zeros). We want to determine whether the product of the two magic numbers can result in a square number, that is,  $X\times Y=R_n^2$. In this section, we respond negatively to this statement.

\begin{lemma}[\cite{costadouglas2}]
If $n>1$, $10^n-1$ is not a square number.
\end{lemma}
\begin{proof}
For $n>1$, $10^n$ is a multiple of $4$, therefore, $10^n-1$ leaves the remainder at $3$ when divided by $4$. The division of a square number by four admits only remainders $0$ or $1$. We conclude that $10^n-1$ is not a square.
\end{proof}

\begin{lemma} [\cite{costadouglas1, niven}]
The number $R_n$ with $n>1$ is not a square.
\end{lemma}
\begin{proof}
It follows from
\begin{eqnarray*}
10^n-1 &=& 999\hdots 99=9\times 111\hdots 11 \\
       & = & 9 \cdot R_n = 3^2 \cdot R_n \ .
\end{eqnarray*}
Because $n\geq 2$, if $R_n$ is a square number, then $10^n-1$ is a square number too.
\end{proof}

\begin{remark}
A non-zero natural number formed by repeating the same digit (digit) is called a single digit, as in previous lemmas for $a=1$ or $9$. An interesting property of a single digit is that no single digit with two or more digits is a square number \cite{costadouglas2, niven}.
\end{remark}

According to \cite{beiler, costadouglas1, costadouglas2}, for all $n \geq 1$, the repunits $R_n$ can be expressed by the sequence

\begin{equation*}
\label{ERb10.1}
R_1=1 \mbox{ and } R_{n+1} = 10\cdot R_n+1 \mbox{ for } n>1 \ .
\end{equation*}

We denote the greatest common divisor between two non-zero natural numbers $a$ and $b$ as $(a,b)$. So,
\begin{lemma}
\label{Lmdc}

Let $m,n$ be natural numbers with $m \geq n>0$ then we have
$(R_m, R_n) = R_{(m,n)}$.
\end{lemma}
\begin{proof}

There are integers $q, r_1$ such that $m=qn+r_1$ with $0\leq r_1 < n$. By applying the Euclidean division of $n$ by $r_1$, we obtain $n=q_1r_1+r_2$. Following this process, we obtain a sequence of remainders $r_1, r_2, \hdots , r_s, r_{s+1} = 0$, with $r_s = (m, n)$. Because  $R_m=10^{r_1}R_n+R_{r_1}$,
$
(R_m, R_n) = (R_n, R_{r_1}) = \hdots = (R_{r_s}
, R_{r_{s+1}} ) = (R_{r_s}, 0) = R_{(m,n)} \
.
$
\end{proof}

\begin{lemma}
\label{LRmn}
Let $m>n>1$ be integers; then, $R_m\cdot R_n$ is not a square number.
\end{lemma}
\begin{proof}
We denote this as $d=(m,n)$. From Lemma \ref{Lmdc} that $(R_m,R_n)=R_d$.  If $d=1$ then $R_m$ and $R_n$ are coprime, so there is no common divisor, which means that $R_m\cdot R_n$ is not a square because $R_m$ and $R_n$ are not square numbers. If $d>1$, $R_d$ is the greatest common divisor of $R_m$ and $R_n$. So, $R_m=R_d\cdot q_1$ and $R_n=R_d\cdot q_2$ with $q_1 \neq q_2$ , $(q_1,q_2)=1$ since $m\neq n$. We have
$R_m\cdot R_n = (R_d)^2\cdot q_1\cdot q_2 \ ,$
which implies that $R_m\cdot R_n$ is also not a square number because $q_1$ and $q_2$ are not square numbers.
\end{proof}

\begin{proposition}
\label{Pprimos} 
Let $m>n>1$ be integers; then, $(10^m-1)(10^n-1)$ is not a square number.
\end{proposition}
\begin{proof} 

We can write $(10^m-1)=9\cdot R_m$. Therefore,
$(10^m-1)(10^n-1)=9^2\cdot R_m \cdot R_n \ .$
It follows from Lemma \ref{LRmn} that $R_m\cdot R_n$ is not a square number, and thus $(10^m-1)(10^n-1)$ is not a square number.
\end{proof}

Finally, we have 
\begin{theorem}
 
The product of these two distinct reverse divisors is never a square.
\end{theorem}
\begin{proof}

From Theorem \cite{webs2} that a reverse divisor has the form $11\times(10^m-1)$ or $22\times(10^m-1)$. From Proposition \ref{Pprimos}, we conclude that the product of distinct reverse divisors is never a square number.
\end{proof}
However, the product of the reverse divisor and its reverse is always a square (see \cite{webs2}).

\section{Conclusion}

Throughout this study, we establish surprising connections among magic numbers, reverse divisors, square numbers, repunits and undulating numbers. Because math is generous, we are sure that we can obtain other interesting properties involving these numbers. We believe that recreational and curious activities play an important role in all sciences and, in particular, can make mathematics seem more natural. We invite the reader to this journey.
\bibliography{sn-bibliography}
\end{document}